\documentclass{article}
\usepackage[utf8]{inputenc}
\usepackage{amsmath, amsthm, amssymb,xcolor}
\usepackage{enumerate, comment}
\usepackage{hyperref}
\usepackage{todonotes}
\usepackage{tikz}

\title{Independent domination versus packing in subcubic graphs}
\author{Eun-Kyung Cho~\thanks{Corresponding author.~\texttt{ekcho2020@gmail.com}~Department of Mathematics, Hankuk University of Foreign Studies, Yongin-si, Gyeonggi-do, Republic of Korea.~Eun-Kyung Cho was supported by Basic Science Research Program through the National Research Foundation of Korea (NRF) funded by the Ministry of Education (No. RS-2023-00244543).} \and Minki Kim~\thanks{\texttt{minkikim@gist.ac.kr}~Division of Liberal Arts and Sciences, GIST, Gwangju, Republic of Korea.~Minki Kim was supported by Basic Science Research Program through the National Research Foundation of Korea (NRF) funded by the Ministry of Education (NRF-2022R1F1A1063424) and by GIST Research Project grant funded by
the GIST in 2023.}}

\date{\today}

\newtheorem{theorem}{Theorem}[section]
\newtheorem{lemma}[theorem]{Lemma}
\newtheorem{corollary}[theorem]{Corollary}
\newtheorem*{claim}{Claim}

\newtheorem{conjecture}[theorem]{Conjecture}
\newtheorem{question}[theorem]{Question}
\newtheorem{observation}[theorem]{Observation}

\def \eun#1 {\textcolor{brown}{Eun: #1}}

\begin{document}
\maketitle

\begin{abstract}
    In 2011, Henning, L\"{o}wenstein, and Rautenbach observed that the domination number of a graph is bounded from above by the product of the packing number and the maximum degree of the graph.
    We prove a stronger statement in subcubic graphs: the independent domination number is bounded from above by three times the packing number.
\end{abstract}

\section{Introduction}
Throughout this paper, we consider only finite graphs.

We say a graph is {\em dominated} by a vertex subset $A$ if every vertex of the graph is either in $A$ or is adjacent to a vertex in $A$. In this situation, we say $A$ is a {\em dominating set} of $G$. The {\em domination number} of $G$ is the size of a minimum dominating set of $G$, and is denoted by $\gamma(G)$. Understanding the domination number of graphs is one of the most fundamental topics in graph theory.

An {\em independent dominating set} of $G$ is a dominating set of $G$ whose vertices are not adjacent to each other, that is, it is an independent set that dominates $G$. Analogously, the {\em independent domination number} of $G$, denoted by $i(G)$, is defined as the size of a minimum independent dominating set of $G$. Independent dominating sets can be understood as maximal independent sets, and the independent domination number is the size of minimum maximal independent set. The independent domination number of graphs has been extensively studied since 1960s~\cite{berge1962theory, ore1962theory}. See \cite{goddard2013independent} for an overview of independent domination numbers of graphs. 
Obviously, $i(G) \geq \gamma(G)$ for every graph $G$, and there has been a number of results comparing $i(G)$ to $\gamma(G)$. See, for example, \cite{cho2023tight,
cho2023independent}.

Meanwhile, the domination number of a graph can be viewed as an analogue of a covering number, the number of spherical balls needed to cover a given space.
To understand this, regard a closed neighborhood of a vertex as a graph analogue of unit ball in the plane.
Covering number has often been studied in comparison with spherical packing number, the number of spherical balls that we can put in a given space without any overlap.
In a similar point of view, we can consider a graph analogue of the packing number.
A {\em packing} (sometimes called a $2$-packing) of a graph $G$ is a vertex set whose closed neighborhoods are pairwise disjoint, or equivalently, whose pairwise distance in $G$ is at least $3$.
A {\em maximal packing} of $G$ is a packing in $G$ that is not properly contained in the larger packing of $G$.
The {\em packing number} of $G$ is the maximum size of a packing in $G$, and is denoted by $\rho(G)$.
A study of a packing related to a dominating set in graphs has a long history, which dates back to 1970s~\cite{harary1970covering, meir1975relations}. See also~\cite{haynes1998fundametals, henning2011dominating, lowenstein2013chiraptophobic, topp1991packing} for an overview of research in this direction.

Our focus is to reveal a relation between the independent domination number and the packing number in graphs.
As the first step toward this direction, we prove the following:
\begin{theorem}\label{thm:cubic-i-rho2}
Let $G$ be a graph where every vertex has degree at most $3$.
Then for every maximal packing $S$ of a graph $G$, there is an independent dominating set of size at most $3|S|$.
\end{theorem}
Theorem~\ref{thm:cubic-i-rho2} was motivated by an observation by Henning, L\"{o}wenstein, and Rautenbach~\cite{henning2011dominating}
that says for every graph $G$ with $\delta(G) \ge 1$, the neighborhood $N_G(S)$ of every maximal packing $S$ of $G$ is a dominating set of $G$. This implies that $\gamma(G) \leq \Delta(G)\rho(G)$ when $G$ is a graph with maximum degree $\Delta(G)$, and they also characterized when the equality holds for $\Delta(G) = 3$: if $G$ is a connected graph with $\Delta(G) \le 3$, then $\gamma(G) =3 \rho(G)$ if and only if $G \in \{H_1, H_2, H_3\}$, where $H_i$'s are the graphs depicted in Figure~\ref{fig:tight}.
Note that $H_1$ and $H_2$ are the only cubic non-planar graphs of order $8$.
The graph $H_2$ is the {\it Wagner graph}, and $H_3$ is the {\it Petersen graph}.
\begin{figure}[htbp]
\centering
\begin{tikzpicture}
\begin{scope}[shift={(-2,0)}]
\filldraw[fill=black, draw=black] \foreach \x in {0,45,90,135,180,225,270,315}{ 
(\x:1) circle (0.05)};
\draw (0,0) circle (1cm);
\draw (0:1) -- (225:1);
\draw (180:1) -- (315:1);
\draw (90:1) -- (270:1);
\draw (45:1) -- (135:1);
\node at (0,-2) {$H_1$};
\end{scope}   

\begin{scope}[shift={(1,0)}]
\filldraw[fill=black, draw=black] \foreach \x in {0,45,90,135,180,225,270,315}{ 
(\x:1) circle (0.05)};
\draw (0,0) circle (1cm);
\draw (0:1) -- (180:1);
\draw (45:1) -- (225:1);
\draw (90:1) -- (270:1);
\draw (135:1) -- (315:1);
\node at (0,-2) {$H_2$};
\end{scope}   

\begin{scope}[shift={(4,0)}]
\filldraw[fill=black, draw=black] \foreach \x in {18,90,162,234,306}{ 
(\x:1) circle (0.05)};
\filldraw[fill=black, draw=black] \foreach \x in {18,90,162,234,306}{ 
(\x:0.6) circle (0.05)};
\draw (0,0) circle (1cm);
\draw \foreach \x in {18,90,162,234,306}{(\x:0.6) -- (\x:1)};
\draw (18:0.6) -- (162:0.6) -- (306:0.6) -- (90:0.6) -- (234:0.6) -- cycle;
\node at (0,-2) {$H_3$};
\end{scope}

\begin{scope}[shift={(6.2,0)}]
\filldraw[fill=black, draw=black] \foreach \x in {-0.8,0,0.8}{ 
(0,\x) circle (0.05)};
\filldraw[fill=black, draw=black] \foreach \x in {-0.8,0,0.8}{ 
(1.5,\x) circle (0.05)};
\draw \foreach \x in {-0.8,0,0.8}{(0,\x) -- (1.5,-0.8)};
\draw \foreach \x in {-0.8,0,0.8}{(0,\x) -- (1.5,0)};
\draw \foreach \x in {-0.8,0,0.8}{(0,\x) -- (1.5,0.8)};
\node at (0.8,-2) {$K_{3,3}$};
\end{scope}
\end{tikzpicture}
    \caption{Tight examples for Corollary~\ref{cor:cubic-i-rho}}
    \label{fig:tight}
\end{figure}
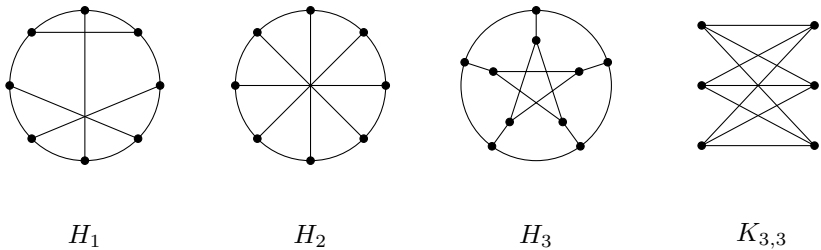

Here is an immediate corollary of the main theorem:
\begin{corollary}\label{cor:cubic-i-rho}
For every subcubic graph $G$, $i(G) \le 3 \rho(G)$. 
\end{corollary}
Observe that the inequality in Corollary~\ref{cor:cubic-i-rho} is also tight by graphs $H_1$, $H_2$, $H_3$ in~\cite{henning2011dominating}. In addition, we observe that $i(G) = 3\rho(G)$ when $G$ is a complete bipartite graph $K_{3,3}$.

Our paper is organized as follows.
In Section~\ref{sec:preliminaries}, we indicate notation and provide an observation and a lemma that are used in the proof of Theorem~\ref{thm:cubic-i-rho2}.
The proof of Theorem~\ref{thm:cubic-i-rho2} is presented in Section~\ref{sec:proof}, and we list some remarks and questions in Section~\ref{sec:future}.

\section{Preliminaries}\label{sec:preliminaries}
For a positive integer $k$, let $[k]$ be the set of all positive integers that are at most $k$.
That is, $[k] = \{1, 2, \ldots, k\}$.
For sets $S$ and $T$ such that $S \cap T = \emptyset$, we use $S \sqcup T$ to denote the disjoint union of $S$ and $T$. 

Let $G$ be a simple and undirected graph. The vertex set and the edge set of $G$ are denoted by $V(G)$ and $E(G)$, respectively.

For $v \in V(G)$, we denote by $N_G(v)$ the {\em neighborhood} of $v$ in $G$ and by $N_G[v]$ the {\em closed neighborhood} of $v$ in $G$, that is, $N_G[v] = N_G(v) \cup \{v\}$.
The {\em degree} of $v$ in $G$ is denoted by $\deg_G(v)$, and the maximum degree and the minimum degree of $G$ are denoted by $\Delta(G)$ and $\delta(G)$, respectively, that is,
$\Delta(G) = \max\{\deg_G(v): v\in V(G)\}$ and $\delta(G) = \min\{\deg_G(v): v\in V(G)\}$.

For $S \subseteq V(G)$, let $N_G(S) = \bigcup_{v \in S}N_G(v)$, and $N_G[S] = \bigcup_{v \in S} N_G[v]$.
Also, we use $G[S]$ to denote the subgraph of $G$ induced by $S$.

The {\em length} of a path in $G$ is the number of edges in the path.
The {\em end vertex} of a path in $G$ is a vertex which is incident with exactly one edge in the path. So each path has exactly two end vertices.
For $u, v \in V(G)$, a {\em shortest path} from $u$ to $v$ in $G$ is a path having minimum length among all paths whose end vertices are $u$ and $v$. The length of such a path is called the {\em distance} between $u$ and $v$ in $G$, and is denoted by $d_G(u,v)$.

Here are simple facts that allow us to consider only connected graphs when proving Theorem~\ref{thm:cubic-i-rho2}.

\begin{observation}\label{obs:connected}
Let $G$ be a graph with two components $G_1$ and $G_2$.
\begin{itemize}
\item[(1)] A set $S$ is a maximal packing in $G$ if and only if $V(G_1) \cap S$ is a maximal packing in $G_1$, and $V(G_2) \cap S$ is a maximal packing in $G_2$.
\item[(2)] A set $I$ is an independent dominating set of $G$ if and only if $V(G_1) \cap I$ is an independent dominating set of $G_1$, and $V(G_2) \cap I$ is an independent dominating set of $G_2$.
\end{itemize}
\end{observation}

In a graph, edges are unordered pairs of vertices. If we replace the edges of a graph with ordered pairs of vertices, it gives us a {\em directed graph}, or a {\em digraph}. The edges of a digraph are called {\em arcs}.
Let $D$ be a digraph with vertex set $V(D)$ and arc set $A(D)$.
For $a, b \in V(D)$, we use $(a,b)$ to denote an arc from $a$ to $b$.
A digraph is {\em antisymmetric} if at most one among $\{(a,b), (b,a)\}$ is in $A(D)$ for each $a,b \in V(D)$.

For $v \in V(D)$, the {\em in-neighbor} of $v$ is a vertex $u$ in $D$ such that $(u,v) \in A(D)$, and the {\em out-neighbor} of $v$ is a vertex $w$ in $D$ such that $(v,w) \in A(D)$.
The {\em in-degree} of $v$ is the number of in-neighbors of $v$, and is denoted by $\deg^{-}_D(v)$.
The {\em out-degree} of $v$ is the number of out-neighbors of $v$, and is denoted by $\deg^{+}_D(v)$.
Note that $\sum_{v \in V(D)} \deg^{-}_D(v) = \sum_{v \in V(D)} \deg^{+}_D(v)$.
A {\em directed path} in $D$ is a path $v_1v_2 \ldots v_k$ in $D$ such that $(v_i,v_{i+1}) \in A(D)$ for all $i \in [k-1]$.
An {\em orientation} $D$ of a graph $G$ is a antisymmetric digraph whose underlying graph is $G$.

We will use the following lemma in the proof of Theorem~\ref{thm:cubic-i-rho2}. This is already known, for example, by Theorem 2 in \cite{farrugia2005orientable}, but we include a proof for the completeness of this manuscript.

\begin{lemma}\label{lem:no_source}
Every connected multigraph with minimum degree at least $2$ admits an orientation with no sources.
\end{lemma}
\begin{proof}
 Let $H$ be a multigraph with minimum degree $2$.
    For an orientation $D$ of $H$, let $s(D)$ be the number of sources of $D$.
    It is sufficient to show that for every orientation $D$ of $H$ with $s(D) \geq 1$, there is another orientation $D'$ of $H$ with $s(D') < s(D)$.
    
    Suppose there is a source, say $v$, of $D$.
    Let $\text{Reach}^+_D(v)$ be the set of all vertices that is an endpoint of a directed path starting from $v$.
    Suppose all vertices in $\text{Reach}^+_D(v)$ has in-degree at most $1$.
    This implies that each of them has out-degree at least $1$.
    Let $K$ be the subgraph of $D$ induced by $X := \{v\} \cup \text{Reach}^+_D(v)$.
    Then,
    \[|X|-1 = |\text{Reach}^+_D(v)| \geq \sum_{x \in X}d^-_K(x) = \sum_{x \in X}d^+_K(x) \geq |X|,\] which is a contradiction.

    Now we may assume that there is $w \in \text{Reach}^+_D(v)$ whose in-degree is at least $2$. Let $P$ be a directed path in $D$ that starts from $v$ and terminates at $w$. We define $D'$ as the orientation of $H$ obtained from $D$ by reversing all orientation of the edges in $P$. 
    Then we have $\deg_{D'}^-(v) = 1$, so $v$ is not a source in $D'$.
    So is $w$; since $\deg_D^-(w) \geq 2$ and $\deg_{D}^+(w) \geq 0$, the above process makes $\deg_{D'}^-(w) \geq 1$ and $\deg_{D'}^+(w) \geq 1$.
    For all vertices other than $v$ and $w$, including all intermediate vertices in $P$, the in-degree and out-degree are preserved.
    This implies that $s(D') < s(D)$ as desired.
\end{proof}

\section{Proof of Theorem~\ref{thm:cubic-i-rho2}}\label{sec:proof}
If $G$ consists of a single vertex, then the statement is obvious.
Thus, we may assume that there are at least two vertices in $G$.
By Observation~\ref{obs:connected}, it is sufficient to show that the statement holds for connected graphs.
Let $G$ be a connected subcubic graph on a finite vertex set $V$, 
and $S$ be a maximal packing of $G$. 
Let $N=N_G(S)$ and $R=V\setminus{N_G[S]}$.
Note that since $S$ is a packing of $G$, every vertex in $N$ has exactly one neighbor in $S$.

We first observe that $N$ is a dominating set of $G$.
Since $G$ is connected, every vertex in $S$ must have a neighbor in $N$.
On the other hand, by the maximality of $S$, every vertex in $R$ must be adjacent to a vertex in $N$. 
Therefore, $N$ dominates the whole graph $G$.
We will modify $N$ to obtain an independent set in $G$ of size at most $3|S|$ that still dominates $G$.

Let $H = G[N]$.
Since every vertex in $N$ has a neighbor in $S$ and has degree at most $3$ in $G$, the induced subgraph $H$ has maximum degree at most $2$.
This implies that $H$ is the disjoint union of cycles $C_1,\ldots,C_p$, paths $P_1,\ldots,P_q$, and isolated vertices.
We may regard the set of $P_i$'s of length $1$ in $H$ as an induced matching $M$.
Let $W$ be the set of endpoints of the edges in $M$, that is, $H[W] = M$.

Given $B \subseteq N$, let $X(B)$ be the set of all vertices $s$ in $S$ such that $|N_G(s)| = 3$ and $N_G(s) \subseteq (N\setminus B) \cap W$. We take a set $A \subseteq N$ that satisfies the following:
\begin{enumerate}[(i)]
\item For each path $P_i$ of length at least $2$, $A$ contains the two endpoints of $P_i$.
\item $A$ is a maximal independent set in $H$ satisfying (i).
\item $|X(A)|$ is minimum among all choices satisfying (i) and (ii).
\end{enumerate}
Note that, by (ii), $A$ contains all isolated vertices of $H$. We claim that the minimality assumption of (iii) actually implies that $X(A)$ is an empty set.

\begin{claim} 
$|X(A)| =0$.
\end{claim}
\begin{proof}
Suppose $|X(A)| \geq 1$. Draw a subcubic multigraph $Q$ on $X(N\setminus W)$ such that $uv \in E(Q)$ if and only if there is an edge $u'v' \in M$ such that $uu', vv' \in E(G)$.
Let $Q'$ be the union of all components of $Q$ that contains a vertex in $X(A)$.
Note that $V(Q')$ is nonempty since $|X(A)| \geq 1$.
See Figure~\ref{fig:Q} for an illustration of the construction of $Q$ and $Q'$.
\begin{figure}[htbp]
\centering
\begin{tikzpicture}
\begin{scope}[shift={(0,0)}, scale=1.2]
\filldraw[fill=white!80!black, draw=white!80!black, rounded corners] (0,-0.4) rectangle (9,0.5){};
\filldraw[fill=white!90!black, draw=white!90!black, rounded corners] (0,-2) rectangle (9,-0.7){};

\node at (9.5,0) {\LARGE$S$};
\node at (9.5,-1.2) {\LARGE$N$};

\node[circle,fill=black, draw=black, inner sep=0.04cm, label={90:$v_1$}] (v1) at (1,0) {};
\node[circle,fill=black, draw=black, inner sep=0.04cm, label={90:$v_2$}](v2) at (2,0) {};
\node[circle,fill=black, draw=black, inner sep=0.04cm, label={90:$v_3$}](v3) at (3,0) {};
\node[circle,fill=black, draw=black, inner sep=0.04cm, label={90:$v_4$}](v4) at (4,0) {};

\node[rectangle,fill=white, draw=black, inner sep=0.05cm](v11) at (0.75,-1) {};
\node[rectangle,fill=white, draw=black, inner sep=0.05cm](v12) at (1,-1) {};
\node[rectangle,fill=white, draw=black, inner sep=0.05cm](v13) at (1.25,-1) {};

\node[rectangle,fill=white, draw=black, inner sep=0.05cm](v21) at (1.75,-1) {};
\node[rectangle,fill=white, draw=black, inner sep=0.05cm](v22) at (2,-1) {};
\node[rectangle,fill=white, draw=black, inner sep=0.05cm](v23) at (2.25,-1) {};

\node[rectangle,fill=white, draw=black, inner sep=0.05cm](v31) at (2.75,-1) {};
\node[rectangle,fill=white, draw=black, inner sep=0.05cm](v32) at (3,-1) {};
\node[rectangle,fill=white, draw=black, inner sep=0.05cm](v33) at (3.25,-1) {};

\node[rectangle,fill=white, draw=black, inner sep=0.05cm](v41) at (3.75,-1) {};
\node[rectangle,fill=white, draw=black, inner sep=0.05cm](v42) at (4,-1) {};
\node[rectangle,fill=white, draw=black, inner sep=0.05cm](v43) at (4.25,-1) {};

\draw (v11) circle (0.15cm);
\draw (v13) circle (0.15cm);
\draw (v31) circle (0.15cm);
\draw (v33) circle (0.15cm);
\draw (v42) circle (0.15cm);
\draw (v43) circle (0.15cm);

\node[rotate around={90:(0:0)}] at (2,0.6) {$\in$};
\node at (2,1) {$X(A)$};

\draw (v1) -- (v11);
\draw (v1) -- (v12);
\draw (v1) -- (v13);

\draw (v2) -- (v21);
\draw (v2) -- (v22);
\draw (v2) -- (v23);

\draw (v3) -- (v31);
\draw (v3) -- (v32);
\draw (v3) -- (v33);

\draw (v4) -- (v41);
\draw (v4) -- (v42);
\draw (v4) -- (v43);

\draw[ultra thick] (v11) -- (v12);
\draw[ultra thick] (v13) -- (v21);
\draw[ultra thick] (v22) to[out=270, in = 270] (v43);
\draw[ultra thick] (v23) -- (v31);
\draw[ultra thick] (v32) to[out=270, in=270] (v42);
\draw[ultra thick] (v33) -- (v41);

\node[circle,fill=black, draw=black, inner sep=0.04cm, label={90:$u_1$}] (u1) at (5,0) {};
\node[circle,fill=black, draw=black, inner sep=0.04cm, label={90:$u_2$}](u2) at (6,0) {};
\node[circle,fill=black, draw=black, inner sep=0.04cm, label={90:$u_3$}](u3) at (7,0) {};
\node[circle,fill=black, draw=black, inner sep=0.04cm, label={90:$u_4$}](u4) at (8,0) {};

\node[rectangle,fill=white, draw=black, inner sep=0.05cm](u11) at (4.75,-1) {};
\node[rectangle,fill=white, draw=black, inner sep=0.05cm](u12) at (5,-1) {};
\node[rectangle,fill=white, draw=black, inner sep=0.05cm](u13) at (5.25,-1) {};

\node[rectangle,fill=white, draw=black, inner sep=0.05cm](u21) at (5.75,-1) {};
\node[rectangle,fill=white, draw=black, inner sep=0.05cm](u22) at (6,-1) {};
\node[rectangle,fill=white, draw=black, inner sep=0.05cm](u23) at (6.25,-1) {};

\node[rectangle,fill=white, draw=black, inner sep=0.05cm](u31) at (6.75,-1) {};
\node[rectangle,fill=white, draw=black, inner sep=0.05cm](u32) at (7,-1) {};
\node[rectangle,fill=white, draw=black, inner sep=0.05cm](u33) at (7.25,-1) {};

\node[rectangle,fill=white, draw=black, inner sep=0.05cm](u41) at (7.75,-1) {};
\node[rectangle,fill=white, draw=black, inner sep=0.05cm](u42) at (8,-1) {};
\node[rectangle,fill=white, draw=black, inner sep=0.05cm](u43) at (8.25,-1) {};

\draw (u11) circle (0.15cm);
\draw (u21) circle (0.15cm);
\draw (u31) circle (0.15cm);
\draw (u33) circle (0.15cm);
\draw (u42) circle (0.15cm);
\draw (u43) circle (0.15cm);

\draw (u1) -- (u11);
\draw (u1) -- (u12);
\draw (u1) -- (u13);

\draw (u2) -- (u21);
\draw (u2) -- (u22);
\draw (u2) -- (u23);

\draw (u3) -- (u31);
\draw (u3) -- (u32);
\draw (u3) -- (u33);

\draw (u4) -- (u41);
\draw (u4) -- (u42);
\draw (u4) -- (u43);

\draw[ultra thick] (u11) -- (u12);
\draw[ultra thick] (u13) -- (u21);
\draw[ultra thick] (u22) to[out=270, in = 270] (u43);
\draw[ultra thick] (u23) -- (u31);
\draw[ultra thick] (u32) to[out=270, in=270] (u42);
\draw[ultra thick] (u33) -- (u41);

\node at (4.5,-2.5) {\LARGE $\Downarrow$};
\end{scope}

\begin{scope}[shift={(4,-5)}]
\filldraw[fill=white!85!black, draw=white!85!black, rounded corners](-1,-2.5) rectangle (2,1){};
\node[circle, fill=black, draw=black, inner sep=0.04cm, label={90:$v_1$}](v1) at (0,0){};
\node[circle, fill=black, draw=black, inner sep=0.04cm, label={-90:$v_2$}](v2) at (0,-1){};
\node[circle, fill=black, draw=black, inner sep=0.04cm, label={-90:$v_3$}](v3) at (1,-1){};
\node[circle, fill=black, draw=black, inner sep=0.04cm, label={90:$v_4$}](v4) at (1,0){};

\draw (v1) -- (v2) -- (v3);
\draw (v2) -- (v4);
\draw (v3) to[out=45,in=-45](v4);
\draw (v3) to[out=135,in=-135](v4);
\draw (0,0.3) circle (0.3cm);
\node[rotate around={-90:(0,0)}] at (0,-1.6) {$\in$};
\node at (0,-2) {$X(A)$};

\node at (-1.5,-1) {\LARGE $Q'$};

\draw[rounded corners] (-2.5,-3) rectangle (5,1.5){};

\node at (-3,-1) {\LARGE $Q$};

\node[circle, fill=black, draw=black, inner sep=0.04cm, label={90:$u_1$}](u1) at (3,0){};
\node[circle, fill=black, draw=black, inner sep=0.04cm, label={-90:$u_2$}](u2) at (3,-1){};
\node[circle, fill=black, draw=black, inner sep=0.04cm, label={-90:$u_3$}](u3) at (4,-1){};
\node[circle, fill=black, draw=black, inner sep=0.04cm, label={90:$u_4$}](u4) at (4,0){};

\draw (u1) -- (u2) -- (u3);
\draw (u2) -- (u4);
\draw (u3) to[out=45,in=-45](u4);
\draw (u3) to[out=135,in=-135](u4);
\draw (3,0.3) circle (0.3cm);
\end{scope}
\end{tikzpicture}
    \caption{An example of construction of $Q$ and $Q'$ from $G$. Edges in $M$ are drawn by thick lines, and the vertices in $A$ are circled.}
    \label{fig:Q}
\end{figure}
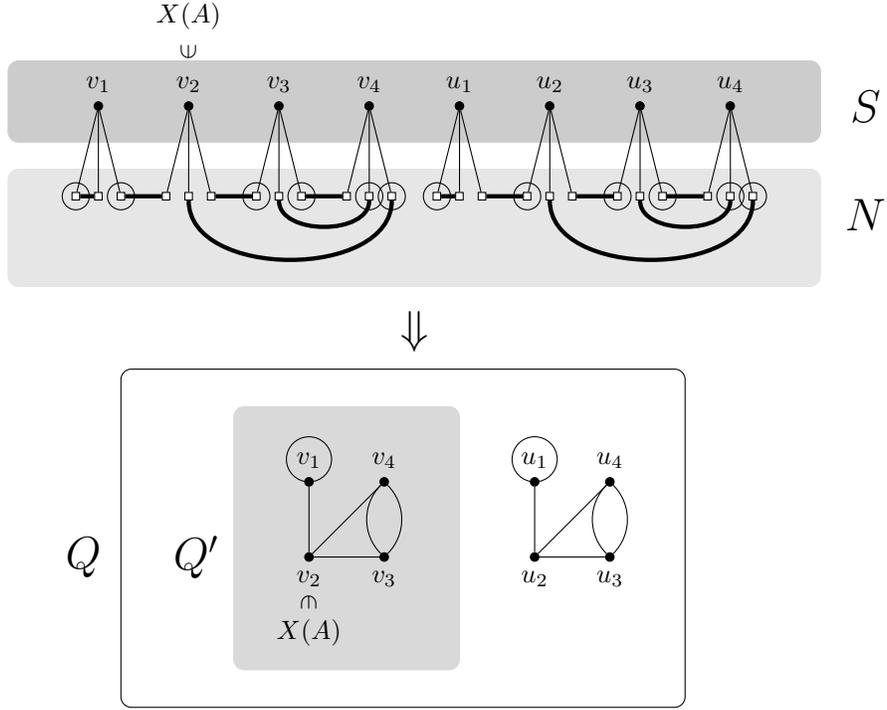

We will show that $Q'$ is a cubic graph to apply Lemma~\ref{lem:no_source}.
Suppose to the contrary that $v \in V(Q')$ and $\deg_{Q'}(v) \le 2$. 
Let $w$ be the vertex in $X(A)$ that has minimum distance to $v$ in $Q$.
Here, we assume $w = v$ if $v \in X(A)$.
Let $w_1 w_2 \ldots w_k$ be the shortest path from $v$ to $w$ in $Q'$, where $w_1 = v$ and $w_k = w$.
Clearly, by the minimality of $k$, each $w_i$ should not be in $X(A)$ for $i\in[k-1]$.
For each $i \in [k-1]$, let $x_iy_i$ be an edge in $M$ such that $x_iw_i, y_iw_{i+1} \in E(G)$.
Consider the path $w_1x_1y_1w_2 \ldots x_{k-1}y_{k-1}w_k$ in $G$.

Recall that, since $v \in X(N\setminus W)$, $v$ has three neighbors in $N$, say $N_G(v) = \{u_1, u_2, x_1\}$, and there are $v_1, v_2 \in N$ such that $u_1v_1, u_2v_2\in M$. 
Since $\deg_{Q'}(v) \le 2$, without loss of generality, we may assume that $v_2$ is not adjacent to a vertex in $X(N\setminus W)$.
Let
\[A' = (A \setminus (\{v_2\}\cup \{x_i : i \in [k-1]\})) \cup \{u_2\}\cup \{y_i : i \in [k-1]\}.\]
See Figure~\ref{fig:Q'} for an illustration of $A'$ obtained from $A$ in Figure~\ref{fig:Q}.
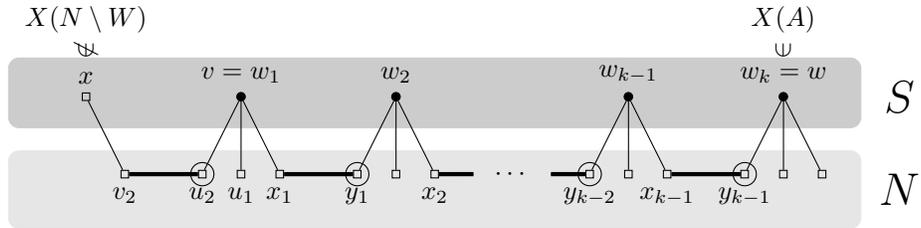
\begin{figure}[htbp]
\centering
\begin{tikzpicture}[scale=1.03]
\filldraw[fill=white!80!black, draw=white!80!black, rounded corners] (-1,-0.4) rectangle (10,0.5){};
\filldraw[fill=white!90!black, draw=white!90!black, rounded corners] (-1,-1.7) rectangle (10,-0.7){};

\node at (10.5,0) {\LARGE $S$};
\node at (10.5,-1.2) {\LARGE $N$};

\node[rectangle, draw=black, inner sep=0.05cm, label={90:$x$}](x) at (0,0){};    
\node[circle, draw=black,fill=black, inner sep=0.04cm, label={90:$v=w_1$}](w1) at (2,0){};
\node[circle, draw=black,fill=black, inner sep=0.04cm, label={90:$w_2$}](w2) at (4,0){};
\node[circle, draw=black,fill=black, inner sep=0.04cm, label={90:$w_{k-1}$}](wk-1) at (7,0){};
\node[circle, draw=black,fill=black, inner sep=0.04cm, label={90:$w_k = w$}](wk) at (9,0){};

\node[rotate around={90:(0,0)}] at (0,0.6) {$\not\in$};
\node at (0,1) {$X(N \setminus W)$};

\node[rectangle, draw=black, inner sep=0.05cm, label={-90:$v_2$}](v2) at (0.5,-1){};

\node[rectangle, draw=black, inner sep=0.05cm, label={-90:$u_2$}](u2) at (1.5,-1) {};
\node[rectangle, draw=black, inner sep=0.05cm, label={-90:$u_1$}](u1) at (2,-1) {};
\node[rectangle, draw=black, inner sep=0.05cm, label={-90:$x_1$}](x1) at (2.5,-1) {};

\node[rectangle, draw=black, inner sep=0.05cm, label={-90:$y_1$}](y1) at (3.5,-1) {};
\node[rectangle, draw=black, inner sep=0.05cm](z2) at (4,-1) {};
\node[rectangle, draw=black, inner sep=0.05cm, label={-90:$x_2$}](x2) at (4.5,-1) {};

\node[rectangle, draw=black, inner sep=0.05cm, label={-90:$y_{k-2}$}](yk-2) at (6.5,-1) {};
\node[rectangle, draw=black, inner sep=0.05cm](zk-1) at (7,-1) {};
\node[rectangle, draw=black, inner sep=0.05cm, label={-90:$x_{k-1}$}](xk-1) at (7.5,-1) {};

\node[rectangle, draw=black, inner sep=0.05cm, label={-90:$y_{k-1}$}](yk-1) at (8.5,-1) {};
\node[rectangle, draw=black, inner sep=0.05cm](zk) at (9,-1) {};
\node[rectangle, draw=black, inner sep=0.05cm](xk) at (9.5,-1) {};

\node[rotate around={90:(0,0)}] at (9,0.6){$\in$}; 
\node at (9,1){$X(A)$};

\draw (x) -- (v2) -- (u2) -- (w1) -- (u1);
\draw (w1)-- (x1) -- (y1) -- (w2) -- (z2);
\draw (w2) -- (x2) -- (5,-1);

\draw[ultra thick] (v2) -- (u2);
\draw[ultra thick] (x1) -- (y1);
\draw[ultra thick] (x2) -- (5,-1);
\draw[ultra thick] (6,-1) -- (yk-2);
\draw[ultra thick] (xk-1) -- (yk-1);

\node at (5.5,-1){$\cdots$};

\draw (6,-1) -- (yk-2) -- (wk-1) -- (zk-1);
\draw (wk-1) -- (xk-1) -- (yk-1) -- (wk) -- (zk);
\draw (wk) -- (xk);

\draw (u2) circle (0.15cm);
\draw (y1) circle (0.15cm);
\draw (yk-2) circle (0.15cm);
\draw (yk-1) circle (0.15cm);
\end{tikzpicture}
    \caption{A path $w_1x_1y_1w_2 \ldots x_{k-1}y_{k-1}w_k$ in $G$ and $A'$. Edges in $M$ are drawn by thick lines, and the vertices in $A'$ are circled.}
    \label{fig:Q'}
\end{figure}

Clearly, $A'$ satisfies (i) and (ii): note that $A'$ is obtained from $A$ by replacing some vertices in $A \cap W$ with the other endpoints in the corresponding edges in $M$.
However, now we have $X(A') = X(A) \setminus \{w\}$, which is a contradiction to the minimality assumption of (iii) of $A$.
Thus we conclude that $Q'$ is a cubic graph, and in particular, $Q'$ is a multigraph having minimum degree at least $2$.
Then by Lemma~\ref{lem:no_source}, there is an orientation $D$ of $Q'$ with no source.
For each $(u,v) \in A(D)$, there is an edge $u'v'$ in $M$ such that $uu', vv' \in E(G)$.
We let $A'' = (A \setminus \{u' : (u,v) \in A(D)\}) \cup \{v' : (u,v) \in A(D)\}$.
It is clear that $A''$ satisfies (i) and (ii) because $A''$ is obtained from $A$ by replacing some vertices in $A \cap W$ with the other endpoints in the corresponding edges in $M$.
On the other hand, since $D$ has no source, $|X(A'')| =0 < 1 \le |X(A)|$, which is again a contradiction to the minimality assumption of (iii) of $A$.
Therefore, it must be $|X(A)| = 0$.
\end{proof}

Now, by the choice of $A$, $N$ is dominated by $A$ in $G$.
We finally modify $A$ to dominate all vertices in $S \cup R$.
Let $T$ be the set of all vertices in $R$ that are not dominated by $A$.
We construct a set $\hat{A}$ by adding to $A$
\begin{enumerate}[(a)]
    \item the set, say $S'$, of all vertices $s \in S$ that are not dominated by $A$, that is, $N(s) \subseteq 
    N\setminus A$, and
    \item an independent dominating set, say $Z$, of $G[T]$.
\end{enumerate}
It is obvious that $S'$ is an independent set.
Clearly, no vertices of $A$ and $Z$ are adjacent to a vertex in $S'$.
Also, by the definition of $Z$, no vertices of $Z$ are adjacent to $A$.
Since $S' \cup A$ dominates $V \setminus T$ and $Z$ dominates $T$, $\hat{A}$ is an independent dominating set of $G$.
We finally claim that $\hat{A}$ is the desired set, that is $|\hat{A}| \leq 3|S|$.

Let $S_i$ be the set of all vertices in $S'$ of degree $i$.
We will show that there is a one-to-one function from $\hat{A}$ to $N \sqcup S_1 \sqcup S_2$, which implies $|\hat{A}| \leq |N|+|S_1|+|S_2|$.
Then, since $|N| \leq |S_1| + 2|S_2| + 3|S_3|$, we have 
\[|\hat{A}| \leq |N| + |S_1| + |S_2| \leq 2|S_1| + 3|S_2| + 3|S_3| \leq 3|S|.\]

Let $s \in S_3$.
By the assumption (i) and (ii) of $A$ and the observation $|X(A)| = 0$, we note that at least one of the neighbors, say $s^*$, of $s$ has two neighbors in $N$.
For each $r \in Z$, there must be a neighbor, say $r^*$, of $r$ that is in $N$.
Now define a function $f: \hat{A} \rightarrow N\sqcup S_1 \sqcup S_2$ by
\[f(v) = \begin{cases} v & \text{ for }v \in A \sqcup S_1 \sqcup S_2 \\ v^* & \text{ for }v \in S_3 \sqcup Z \end{cases}\]
Note that for every pair of $r^*$ and $s^*$, they do not belong to $A$ and that they cannot be the same since $s^*$ does not have a neighbor in $R$.
Thus, it is clear that $f$ is a well defined one-to-one function.
This completes the proof.

\section{Remark}\label{sec:future}
We have investigated how small the minimum independent number can be when the packing number is given in subcubic graphs.
The first question we can ask is a generalization of Theorem~\ref{thm:cubic-i-rho2}, as an analogue of the observation by Henning, L\"{o}wenstein and Rautenbach that $\gamma(G) \leq \Delta(G)\rho(G)$ for every graph.
\begin{question}\label{que:gen}
    Is $i(G) \le \Delta(G) \rho(G)$ for every graph $G$?
\end{question}
For $\Delta(G) \le 2$, Question~\ref{que:gen} is obviously true, and for $\Delta(G) = 3$, it is answered to be true by Theorem~\ref{thm:cubic-i-rho2}.
Thus, the first interesting case is when $\Delta(G) = 4$.

Another interesting question is to characterize all graphs where Theorem~\ref{thm:cubic-i-rho2} is tight.
As we observed in the introduction, $H_1, H_2, H_3$ and $K_{3,3}$ in Figure~\ref{fig:tight} has independent domination number exactly three times the packing number.
However, we do not know whether they are the only examples for the tightness of Theorem~\ref{thm:cubic-i-rho2}.
\begin{question}\label{que:tightness}
If $G$ is a connected graph with $\Delta(G) \le 3$, then for which graphs $i(G) = 3 \rho(G)$ hold?
\end{question}

Once we know the answer for Question~\ref{que:tightness}, it is worth to ask if the ratio $i(G)/\rho(G)$ becomes strictly smaller than $3$ for subcubic graphs if we exclude those satisfying $i(G) = 3 \rho(G)$.
Especially, this is related to the following conjecture by Henning, L\"{o}wenstein, Rautenbach~\cite[Conjecture 3]{henning2011dominating}:

\begin{conjecture}\label{conj:2rho}
    Every connected subcubic graph $G$ except the three graphs $H_1, H_2, H_3$ satisfies $\gamma(G) \le 2 \rho(G)$.
\end{conjecture}
It was shown in \cite{henning2011dominating} that the conjecture is true for {\em claw-free graphs}, the graphs with no induced subgraph isomorphic to $K_{1,3}$.
Since the independent domination number and the domination number are the same in claw-free graphs, so it immediately follows that $i(G) \le 2\rho(G)$ for subcubic claw-free graphs.
As a new step to confirm Conjecture~\ref{conj:2rho}, we can consider subcubic graphs excluding all graphs that belongs to the answer for Question~\ref{que:tightness}.
It is quite ambitious, but we can also try to figure out whether $i(G) \leq 2\rho(G)$ for such graphs, as a stronger analogue of Conjecture~\ref{conj:2rho}.

\section*{Acknowledgement}
Part of this research was conducted during the Winter Workshop in Combinatorics that was held in South Korea from January 30, 2023 to February 3, 2023, organized by Ilkyoo Choi, Minki Kim, and Boram Park.

\bibliographystyle{plain}
\bibliography{ref}

\end{document}